\documentclass{article}

\newcommand{\AMMAuthorName}{Rui Viana}
\newcommand{\AMMBibliographyFiles}{amm-references}
\newcommand{\AMMPostStyleSetup}{\pagestyle{plain}}

\usepackage{maa-monthly}

\final
\providecommand{\AMMPostStyleSetup}{}
\AMMPostStyleSetup
\theoremstyle{plain}
\newtheorem{theorem}{Theorem}
\newtheorem{proposition}[theorem]{Proposition}

\newtheorem*{maintheorem}{Theorem}

\theoremstyle{definition}

\newcommand{\E}{\mathbb E}
\newcommand{\Prb}{\mathbb P}

\providecommand{\AMMTitlePage}{}
\providecommand{\AMMAuthorName}{}
\providecommand{\AMMPreprintDisclosure}{}
\providecommand{\AMMNamedDeclarations}{}
\providecommand{\AMMAuthorBiography}{}
\providecommand{\AMMBibliographyFiles}{amm-references}

\newcommand{\AMMDisplayTitle}{A Common Structure Behind\\Birth Stopping Rules}

\begin{document}

\AMMTitlePage

\title{\AMMDisplayTitle}
\author{\AMMAuthorName}
\maketitle
\mark{{}{Birth Stopping Rules}}

\begin{abstract}
We study stopping rules in which each family is assigned a positive integer and stops the first time its number of girls equals that integer plus a fixed nonnegative integer multiple of its number of boys. After all families have stopped, we ask for the expected proportion of boys and the expected boy-to-girl ratio in the combined population, and show that both expectations depend on the assigned integers only through their sum. We obtain exact series for both expectations, allowing the probabilities of a boy and a girl to be unequal. Our formulas recover the known formulas for stopping at the first girl and for stopping when girls first outnumber boys. These expectations can also be recovered from earlier general results on Galton--Watson degree profiles and shifted inverse moments, but the resulting series do not appear to have been recorded in these forms. To our knowledge, our two elementary proofs, one bijective and one probabilistic, are also new.
\end{abstract}

\vspace{10px}

\vspace{6px}

\section{A puzzle and three related sums.}

Suppose each of \(n\) families continues having children until a girl is born, with boys and girls equally likely at each birth. After all families have stopped, what is the expected proportion of boys among all the children? A tempting first answer is \(1/2\). In a 2010 post on his blog \emph{The Big Questions} \cite{Landsburg}, Steven Landsburg posed this problem and, in follow-up posts, pointed out that \(1/2\) is the ratio of expectations, whereas the expected proportion of boys is given by the expectation of a ratio. Zare addressed the question further in \cite{Zare}.

Let \(B\) and \(G\) denote the total numbers of boys and girls born by the time all \(n\) families have stopped. The expected proportion of boys is
\begin{equation}
\label{eq:first-girl-proportion}
 \E\left(\frac{B}{B+G}\right)
 =
 \frac{n}{n+1}-\frac{n}{n+2}+\frac{n}{n+3}-\cdots.
\end{equation}

This follows from results of Zare \cite{Zare} and Griffiths \cite{Griffiths}. For \(n=1\), this yields \(1-\log 2\). By contrast, the expected boy-to-girl ratio is simpler: since \(G=n\) and \(\E(B)=n\), \(\E\left(B/G\right)=1\).

A closely related problem asks what happens when each family stops when girls first outnumber boys. Under that stopping rule, the expected proportion of boys is
\begin{equation}
\label{eq:girl-majority-proportion}
 \E\left(\frac{B}{B+G}\right)
 =
 \frac{n}{n+2}-\frac{n}{n+4}+\frac{n}{n+6}-\cdots,
\end{equation}
which follows from taking the complement of the stopped win proportion in \cite{BrussPaindaveine}, while the expected boy-to-girl ratio is
\begin{equation}
\label{eq:girl-majority-ratio}
\E\left(\frac{B}{G}\right)
 =
 \frac{n}{2(n+1)}
 +\frac{n}{2^2(n+2)}
 +\frac{n}{2^3(n+3)}+\cdots.
\end{equation}

For \(n=1\), these yield \(1-\pi/4\) \cite{GerholdHubalek,Propp} and \(2\log 2-1\) \cite{GerholdHubalek}, respectively.

The remarkable resemblance among these series hints at a shared structure. We extend these known special cases to a more general stopping rule, and obtain exact series for both expectations. Our proofs use two elementary methods. The first uses bijections on lattice paths and ordered trees to derive recurrences whose iteration yields the series. The second uses random priorities to give parallel probabilistic derivations of the same series.

\subsection{From the puzzle to the general rule.}

Fix a nonnegative integer weight \(w\), and assign family \(i\) a positive integer threshold \(k_i\). Under the \emph{weighted stopping rule}, family \(i\) stops the first time its number of girls equals \(w\) times its number of boys plus \(k_i\). Let \(B_i\) and \(G_i\) denote the numbers of boys and girls in family \(i\) when it stops. Then \(G_i = wB_i+k_i\). When \(w=0\) and \(k_i=1\), a family stops at its first girl. When \(w=1\) and \(k_i=1\), it stops when girls first outnumber boys.

Let \(B\) and \(G\) be the combined totals after all families have stopped, and put \(k=\sum_{i=1}^n k_i\). Then \(G=wB+k\). Assume that all births are independent, with \(b=\Prb(\text{boy})\) and \(g=\Prb(\text{girl})=1-b\), and write \(\rho=b/g\) for the boy-to-girl odds. We work throughout in the regime \(g\geq wb\), in which termination occurs almost surely, with \(g>0\) when \(w=0\).\footnote{Under the tree encoding in Section~\ref{subsec:cut-and-glue}, a family with threshold \(k_i\) corresponds to a forest of \(k_i\) Galton--Watson trees, each having offspring number \(0\) with probability \(g\) and \(w+1\) with probability \(b\). Since \(k_i\) is finite, the standard extinction criterion gives almost-sure termination exactly when \((w+1)b\leq1\), equivalently \(g\geq wb\), apart from the degenerate case \(w=0\), \(b=1\); see \cite[Chapter~I]{AthreyaNey}.}

For each family, start a score at \(0\), give a girl score \(+1\), a boy score \(-w\), and stop when the running score first reaches \(k_i\). Family \(i\)'s score path can be decomposed into \(k_i\) pieces by cutting it when it first visits \(1,\ldots,k_i\), as illustrated in the left panel of Figure~\ref{fig:marked-flip}. Applying this decomposition to every family and subtracting the starting score from each piece produces \(k\) independent paths, each distributed as a family history with threshold \(1\). Their combined totals therefore have the same joint distribution as \((B,G)\). We use this representation in the proofs below, which suffices to prove the theorem for arbitrary positive thresholds \(k_1,\ldots,k_n\).

\begin{maintheorem}
Under the model above, the joint distribution of \(B\) and \(G\) depends on \(k_1,\ldots,k_n\) only through \(k\). Define
\[
P_{k,w}(b)=\E\left(\frac{B}{B+G}\right),
\qquad
R_{k,w}(b)=\E\left(\frac{B}{G}\right).
\]

\begin{enumerate}
\item If \(b\leq1/2\), then
\begin{equation}
\label{eq:proportion-series}
 P_{k,w}(b)
 =
 k\sum_{j\geq1}
 \frac{(-1)^{j-1}\rho^j}{k+j(w+1)}.
\end{equation}

\item If \(w\geq1\), then
\begin{equation}
\label{eq:ratio-series}
 R_{k,w}(b)
 =
 k\sum_{j\geq1}\frac{b^j}{k+jw}.
\end{equation}
\end{enumerate}
\end{maintheorem}

For \(w=0\), \(G=k\) and \(\mathbb E(B)=k\rho\), so \(R_{k,0}(b)=\rho\). When all \(k_i=1\), so that \(k=n\), setting \((w,b)=(0,\tfrac12)\) and \((w,b)=(1,\tfrac12)\) in \eqref{eq:proportion-series} recovers \eqref{eq:first-girl-proportion} and \eqref{eq:girl-majority-proportion}. Setting \((w,b)=(1,\tfrac12)\) in \eqref{eq:ratio-series} recovers \eqref{eq:girl-majority-ratio}.

Both series can also be recovered from general generating-function and negative-moment results; we discuss these connections in Section~\ref{sec:historical-context}.

\subsection{Alternative interpretations.}

This general stopping rule has an interpretation in terms of a simple branching model for a disease outbreak. Suppose outbreak \(i\) begins with \(k_i\) infected individuals. Each infected individual either causes no secondary infections, with probability \(g\), or infects \(w+1\) new individuals, with probability \(b\). The outbreak ends when there are no active infections remaining. Within this framing, the quantities studied in this paper are the expected proportion of infected individuals who transmit the disease and the expected ratio of transmitting to nontransmitting individuals.

The same model describes analogous processes in other settings, such as recursive algorithms in which each call either terminates with probability \(g\) or makes \(w+1\) recursive calls with probability \(b\), and queueing systems with batch arrivals.

\section{Marked bijections.}

\subsection{Flipping a marked child.}

\begin{proposition}
\label{prop:flip}
\begin{equation*}
 P_{k,w}(b)
 =
 \rho\,
 \frac{k}{k+w+1}
 \bigl(1-P_{k+w+1,w}(b)\bigr).
\end{equation*}
\end{proposition}

\begin{proof}
After the \(k\) families have stopped, choose one child uniformly from the combined population and mark that child. \(P_{k,w}(b)\) is the probability that the marked child is a boy.

Concatenate the birth histories of the \(k\) families and form a score path using the construction above. This path first reaches level \(j\) at the end of the \(j\)th family history. In particular, it first reaches \(k\) at its final step, and the individual histories are recovered by cutting the path at its first visits to \(1,2,\ldots,k\).

If the marked child is a boy, flip it into a girl. This creates the following bijection, also illustrated in Figure~\ref{fig:marked-flip}:
\[
\left\{
\begin{array}{c}
\text{\(k\) family histories with} \\
\text{a marked boy}
\end{array}
\right\}
\overset{\text{flip}}{\longleftrightarrow}
\left\{
\begin{array}{c}
\text{\(k+w+1\) family histories} \\
\text{with a marked girl in one of} \\
\text{the first \(k\) blocks}
\end{array}
\right\}.
\]

\begin{figure}[htbp]
\centering
\includegraphics{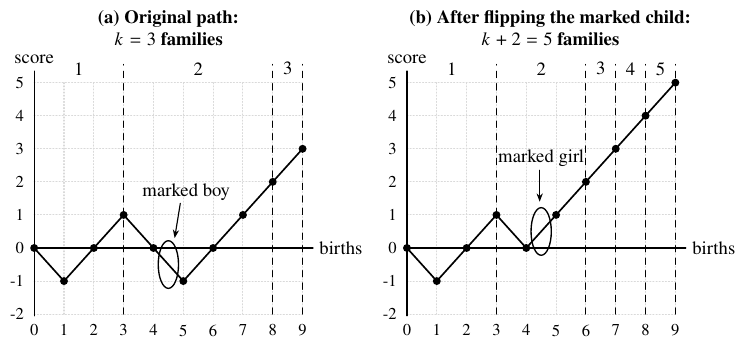}
\caption{Family cuts and the flip bijection for \(w=1\). Flipping one boy to a girl raises the suffix by \(2\) and turns \(k\) family histories into \(k+2\).}
\label{fig:marked-flip}
\end{figure}

Before the marked child, the new path agrees with the original, and after that it is higher by \(w+1\). Since the original path first reaches \(k\) at the end, the new path first reaches \(k+w+1\) at the end and therefore it is a valid concatenation of \(k+w+1\) family histories. Further, the score immediately after the marked child's birth is at most \(k\), and level \(k\) has not been attained previously, so the marked girl lies in one of the first \(k\) families.

Conversely, flipping such a marked girl back to a boy lowers the path after that step by \(w+1\), and the resulting path first reaches \(k\) at its final step. Thus the flip is a bijection.

The probability of the original marked outcome is \(b^B g^G/(B+G)\). This is \(b/g=\rho\) times the probability of the image, which contains \(B-1\) boys and \(G+1\) girls. Further, by exchangeability of the \(k+w+1\) family blocks, the marked girl is equally likely to lie in any of them, so the probability that it lies in one of the first \(k\) blocks is
\[
 \frac{k}{k+w+1}
 \Prb(\text{the marked child is a girl})
 =
 \frac{k}{k+w+1}
 \bigl(1-P_{k+w+1,w}(b)\bigr).
\]
Multiplying by \(\rho\) proves Proposition~\ref{prop:flip}.

After \(m\) iterations of the recursion,
\begin{equation*}
 P_{k,w}(b)
 =
 k\sum_{j=1}^{m}
 (-1)^{j-1}\frac{\rho^j}{k+j(w+1)}
 +
 (-1)^m\frac{k\rho^m}{k+m(w+1)}
 P_{k+m(w+1),w}(b).
\end{equation*}
Since \(0\leq P_{k+m(w+1),w}(b)\leq1\) and \(b\leq1/2\), \(\rho\leq1\) and the remainder term above tends to zero as \(m\to\infty\). This yields \eqref{eq:proportion-series}.
\end{proof}

\subsection{Cutting at a marked boy.}
\label{subsec:cut-and-glue}

\begin{proposition}
\label{prop:cut-glue}
\begin{equation*}
 R_{k,w}(b)
 =
 b\,\frac{k}{k+w}\bigl(1+R_{k+w,w}(b)\bigr).
\end{equation*}
\end{proposition}

\begin{proof}
Encode each family history as an ordered tree. If the first child is a girl, the tree consists of a single leaf. If the first child is a boy, the running score falls from \(0\) to \(-w\). Divide the remainder of the history at its first visits to
\[
 -w+1,-w+2,\ldots,0,1.
\]
This produces \(w+1\) consecutive subhistories. Make the initial boy an internal vertex and attach the trees corresponding to these \(w+1\) subhistories as its ordered child subtrees. Thus girls correspond to leaves, boys correspond to internal vertices with \(w+1\) ordered children, and \(k\) family histories correspond to a forest of \(k\) full ordered \((w+1)\)-ary trees, as illustrated in Figure~\ref{fig:history-tree}.

\begin{figure}[htbp]
\centering
\includegraphics{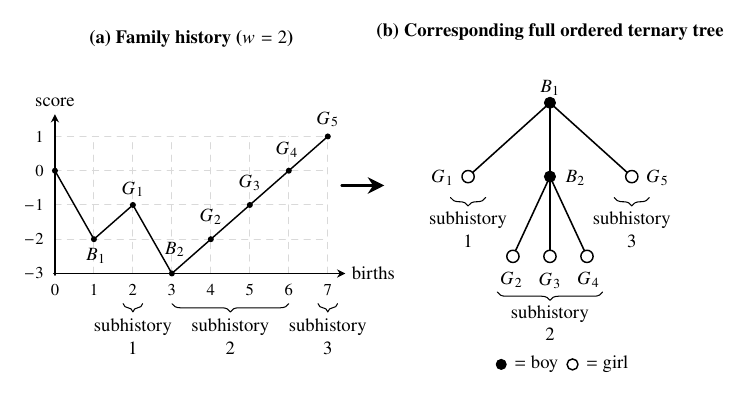}
\caption{The recursive correspondence between a \(w=2\) family history and a full ordered ternary tree. Boys become internal vertices and girls become leaves. The three subhistories following the initial boy become the three ordered child subtrees. Here, the subscripts on \(B_j\) and \(G_j\) index individual children, not families.}
\label{fig:history-tree}
\end{figure}

Choose a boy vertex \(v\) and mark it. Let \(T_0,T_1,\ldots,T_w\) be its ordered child subtrees. The bijection illustrated in Figure~\ref{fig:cut-and-glue} has three steps: delete \(v\) and replace the subtree rooted at \(v\) by \(T_0\); append \(T_1,\ldots,T_w\), in that order, to the end of the forest; and mark the root of \(T_0\). The result is a forest of \(k+w\) trees with a marked vertex, not necessarily a boy, in one of the first \(k\) trees.

To reverse the construction, remove the final \(w\) trees. In the component containing the marked vertex \(x\), replace the subtree rooted at \(x\) by a new boy vertex whose ordered child subtrees are the original subtree at \(x\), followed, in order, by the \(w\) removed trees.
\[
\left\{
\begin{array}{c}
\text{\(k\) trees with} \\
\text{a marked boy}
\end{array}
\right\}
\overset{\text{cut/glue}}{\longleftrightarrow}
\left\{
\begin{array}{c}
\text{\(k+w\) trees} \\
\text{with a marked vertex in one of} \\
\text{the first \(k\) trees}
\end{array}
\right\}.
\]

\begin{figure}[htbp]
\centering
\includegraphics{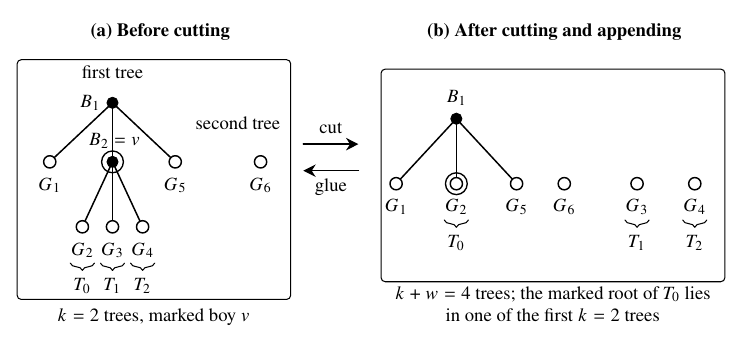}
\caption{The cut-and-glue operation for \(w=2\). Deleting the marked boy \(B_2=v\) leaves \(T_0\) in place, with its root marked, and appends \(T_1\) and \(T_2\) as the last two components. Here, the subscripts on \(B_j\) and \(G_j\) index individual children, not families.}
\label{fig:cut-and-glue}
\end{figure}

We can interpret \(R_{k,w}(b)\) as the total weight of all pairs consisting of a forest and a marked boy vertex, assigning weight \(b^B g^G/G\) to a pair with \(B\) boys and \(G\) girls. Under the bijection, the original weight is \(b\) times the image weight, as the image has \(B-1\) boys and the same \(G\) girls.

On the image side, let \(V_i\) be the number of vertices in the \(i\)th tree. Since the marked vertex must lie in one of the first \(k\) trees, and the \(k+w\) trees are exchangeable, the total weight of the image objects is
\begin{equation*}
  \frac{k}{k+w}
 \E\left(\frac{\sum_{j=1}^{k+w} V_j}{G}\right)
 =
 \frac{k}{k+w}
 \E\left(\frac{B+G}{G}\right)
 =
 \frac{k}{k+w}
 \bigl(1+R_{k+w,w}(b)\bigr).
\end{equation*}
Multiplying by the weight ratio \(b\) proves Proposition~\ref{prop:cut-glue}.

After \(m\) iterations of the recursion,
\[
 R_{k,w}(b)
 =
 k\sum_{j=1}^{m}\frac{b^j}{k+jw}
 +
 \frac{k b^m}{k+mw}R_{k+mw,w}(b).
\]

For \eqref{eq:ratio-series}, the theorem assumes \(w\geq1\). If \(b=0\), the result is immediate. Otherwise \(0<b<1\). Since every such forest satisfies \(G-wB>0\), we have \(B/G<1/w\), and hence \(R_{k+mw,w}(b)<1/w\). Therefore the remainder term is at most \(k b^m / (w(k+mw))\), which tends to zero as \(m\to\infty\). This yields \eqref{eq:ratio-series}.

\end{proof}

\subsection{A forest bijection for Proposition~\ref{prop:flip}.}

Given a forest of \(k\) trees with a marked boy vertex \(v\), replace \(v\) by a marked leaf and append its \(w+1\) ordered child subtrees to the forest. The result has \(k+w+1\) trees, with the marked leaf in one of the first \(k\). The inverse changes the marked leaf back into a boy vertex and attaches the final \(w+1\) trees to it in order. The bijection preserves the number of vertices and changes the probability weight by \(b/g=\rho\), so exchangeability recovers Proposition~\ref{prop:flip}. Compare this with Proposition~\ref{prop:cut-glue}: there \(T_0\) takes the boy's place, with its root marked, and only \(T_1,\ldots,T_w\) are appended, producing \(k+w\) rather than \(k+w+1\) trees.

\section{Random priorities.}

We now give probabilistic proofs of the two series in the theorem.

\subsection{Priorities for all children.}

Give every child an independent uniform priority in \((0,1)\). Conditional on the completed histories, each child is equally likely to have the largest priority. Hence
\[
 P_{k,w}(b)
 =
 \Prb(\text{the overall winner is a boy}).
\]

Consider a single family, and let \(N_i=B_i+G_i=(w+1)B_i+1\). Let \(Y_i\) be the largest priority in this family. Each priority is at most \(y\) with probability \(y\), so define
\[
F(y)=\Prb(Y_i\leq y)=\E(y^{N_i}).
\]

Conditioning on \(Y_i\), write \(p(Y_i)\) for the probability that the family winner is a boy. We compute \(p\) using the tree encoding from Section~\ref{subsec:cut-and-glue}. If the winning priority is \(y\), then once the winning child is fixed, the contribution from the rest of the tree, excluding the winner and its descendants, is the same whether the winner is a girl or a boy. A girl winner contributes \(g\), whereas a boy winner contributes \(b\), and each of its \(w+1\) child subtrees must have maximum priority below \(y\), contributing \(F(y)^{w+1}\). Hence
\begin{equation}
\label{p-rec}
p(y) = \frac{bF(y)^{w+1}}{g+bF(y)^{w+1}}.
\end{equation}

\(F(y)\) is continuous and strictly increasing, so for \(0 < x < 1\), \(\Prb(F(Y_i)\leq x) = \Prb\bigl(Y_i\leq F^{-1}(x)\bigr) = x\). Thus \(F(Y_i)\) is uniform on \((0,1)\). The overall winner belongs to the family with the largest winning priority \(M=\max(Y_1,\ldots,Y_k)\). Therefore \(P_{k,w}(b) = \E\bigl(p(M)\bigr)\). Expanding \eqref{p-rec} at \(y=M\) and letting \(U=F(M)\) gives

\begin{equation}
\label{eq:P-U-series}
 P_{k,w}(b)
 =
 \E\left(
 \frac{bU^{w+1}}
      {g+bU^{w+1}}
 \right)
 =
\E\left(\sum_{j\geq1}(-1)^{j-1}
 \rho^{j} U^{j(w+1)}\right).
\end{equation}

This series converges, and we may take expectations term by term.\footnote{Since \(b\leq1/2\), \(\rho\leq1\). If \(\rho<1\), the sum of the absolute values is at most the convergent geometric series \(\sum_{j\geq1}\rho^j\). If \(\rho=1\), the partial sums are bounded by \(1\) and converge almost surely because \(U<1\).}

\(U\) is the maximum of \(k\) independent uniform \((0,1)\) random variables. For \(0\leq u\leq1\), \(\Prb(U\leq u)=u^k\), so \(U\) has density \(ku^{k-1}\). Hence, for \(r\geq0\),
\begin{equation}
\label{eq:uniform-moment}
 \E(U^r)
 =
 \int_0^1 u^rku^{k-1}\,du
 =
 \frac{k}{k+r}.
\end{equation}
Using \eqref{eq:uniform-moment} in \eqref{eq:P-U-series} yields \eqref{eq:proportion-series}.

\subsection{Priorities for girls.}

Assume \(w\geq1\) and assign independent uniform priorities only to the girls. In family \(i\), distinguish the \(B_i\) leftmost leaves. Across all \(k\) families, there are therefore \(B\) distinguished girls among \(G\) girls. Conditional on the completed histories, each girl is equally likely to have the largest priority. Hence
\[
 R_{k,w}(b)
 =
 \Prb(\text{the overall winner is distinguished}).
\]

For one family, let \(Y_i\) be its largest girl priority and define
\[
H(y)=\Prb(Y_i\leq y)=\E(y^{G_i}).
\]

Conditioning on \(Y_i\), write \(q(Y_i)\) for the probability that the family winner is distinguished. For fixed \(B_i\) and \(G_i\), the probability that the winner is distinguished and has priority at most \(y\) is \(B_i/G_i y^{G_i}\), whereas \(\Prb(Y_i\leq y\mid B_i,G_i)=y^{G_i}\). Differentiating with respect to \(y\) and averaging over the family history gives

\begin{equation}
\label{q-densities}
q(y)
=
\frac{\E(B_i y^{G_i-1})}{\E(G_i y^{G_i-1})}
=
\frac{\E(B_i y^{G_i})}{\E(G_i y^{G_i})}.
\end{equation}

Using the tree encoding from Section~\ref{subsec:cut-and-glue}, interpret \(A(y)=\E(B_i y^{G_i})\) as the total weight of family trees with one boy marked and \(Y_i \leq y\). If the root is marked, all \(w+1\) child subtrees contribute \(H(y)\). Otherwise, one of them contributes \(A(y)\) and the remaining \(w\) contribute \(H(y)\). Hence
\[
A(y)
=
bH(y)^w\bigl(H(y)+(w+1)A(y)\bigr).
\]

Since \(G_i=wB_i+1\), the expression in parentheses is \(\E(G_i y^{G_i})+A(y)\). Substituting this above, rearranging, and using \eqref{q-densities}, we obtain
\[
q(y)
=
\frac{bH(y)^w}{1-bH(y)^w}.
\]

As in the preceding subsection, the \(H(Y_i)\) are independent uniform random variables. The overall winner belongs to the family with the largest winning priority \(M=\max(Y_1,\ldots,Y_k)\), so \(R_{k,w}(b)=\E(q(M))\). Expanding the expression for \(q(M)\) and letting \(U=H(M)\) gives
\begin{equation}
\label{eq:R-U-series}
R_{k,w}(b)
=
\E\left(
\frac{bU^w}{1-bU^w}
\right)
=
\E\left(
\sum_{j\geq1}b^jU^{jw}
\right).
\end{equation}

This series has nonnegative terms, so we may take expectations term by term. Since \(U\) is again the maximum of \(k\) independent uniform \((0,1)\) random variables, applying \eqref{eq:uniform-moment} to \eqref{eq:R-U-series} yields \eqref{eq:ratio-series}.

\section{Historical context.}
\label{sec:historical-context}

Sex-based stopping rules have long been studied through expected family sizes and population sex ratios. Robbins related them to fair gambling systems \cite{Robbins}, while later work considered first-girl, first-boy, and capped policies through expected counts and ratios of expectations \cite{Paseau,GrechJamesLauri}. Expected proportions appear in Griffiths and Zare for the first-girl rule \cite{Zare,Griffiths}, and in Bruss and Paindaveine for the girl-majority rule \cite{BrussPaindaveine}. Yamaguchi derived a recurrence for the expected proportion of boys in families that stop after a prescribed number of boys. After exchanging boys and girls, his recurrence is the \(w=0\) case of Proposition~\ref{prop:flip} \cite{Yamaguchi}. Related random-walk treatments include Gerhold and Hubalek \cite{GerholdHubalek} and Propp \cite{Propp}.

Related unweighted decompositions of marked trees appear elsewhere. Chen, Li, and Shapiro observe that a plane tree with a marked vertex splits into a leaf-marked tree and the subtree rooted at that vertex \cite{ChenLiShapiro}. Cox uses the \(w=k=1\) case of the tree flip, splitting a binary tree at a marked internal vertex into three binary trees, with a marked leaf in the tree containing the original root \cite{Cox}. Neither paper assigns probability weights to these constructions or derives the forest recurrences used here.

\subsection{The weighted stopping rule and its distribution.}

The weighted boundary itself is classical. Mohanty's coin-tossing game stops the first time the number of heads equals \(r\) times the number of tails plus \(a\) \cite{Mohanty}. Interpreting heads as girls and tails as boys, and setting \((r,a)=(w,k)\), gives our rule for one family. Mohanty derived the probability generating function for the duration, along with a formula counting the paths that avoid the boundary. In our notation, \(B\) has the generalized negative-binomial distribution of Jain and Consul, whose convolution property explains the dependence on the assigned integers only through their sum \cite{JainConsul}. The stopping time \(B+G=k+(w+1)B\) is also a skip-free first-passage time \cite{Steutel} and the total progeny of a Galton--Watson forest \cite{Dwass}. The path and tree encodings used in our proofs are therefore classical \cite{Stanley}. Our formulas come from augmenting those encodings by marking a child or vertex and applying the flip and cut-and-glue bijections.

More generally, Minami gives a multivariate generating function for the total progeny and outdegree counts of a general Galton--Watson tree \cite{Minami}. After specializing the offspring distribution to \(\{0,w+1\}\), passing to a forest of \(k\) independent trees, differentiating with respect to the degree markers, and integrating the resulting expressions, we can also recover our two expectations from his generating function. Minami does not evaluate either of these expectations or state either resulting series.

The series can also be recovered from the shifted inverse moments studied by Kumar and Consul \cite{KumarConsul}. Let \(s_1=k/(w+1)\), and, when \(w\geq1\), let \(s_2=k/w\). Then
\begin{align*}
P_{k,w}(b)&=\frac{s_1}{k}\left(1-s_1\E\left(\frac1{B+s_1}\right)\right),\\
R_{k,w}(b)&=\frac{s_2}{k}\left(1-s_2\E\left(\frac1{B+s_2}\right)\right).
\end{align*}
Identifying their parameters \((m,\beta,\theta,k)\) with our \((k,w+1,b,s)\), respectively, and applying their equation (3.15) with our \(s\) equal to \(s_1\) or \(s_2\) yields the two series after simplification. Their result treats more general negative moments, but neither series appears to be stated there, and their proof uses a differential recurrence rather than the marked bijections or random priorities used here.

\AMMPreprintDisclosure

\section*{Generative AI Disclosure.}

During the development and revision of this article, the author used OpenAI's GPT-5.6 Sol to explore ideas, identify potentially relevant literature, and obtain feedback on the article's structure and exposition. The author checked the tool's terms of use, independently verified all mathematical arguments and bibliographic information, reviewed and revised all AI-assisted text, and takes full responsibility for the originality, accuracy, and integrity of the article.

\AMMNamedDeclarations

\bibliographystyle{vancouver}
\bibliography{\AMMBibliographyFiles}

\begin{thebibliography}{10}

\bibitem{Landsburg}
Landsburg SE.
\newblock Are You Smarter Than {Google}? [blog post on the Internet]. The Big
  Questions; 2010 Dec 21 [cited 2026 Aug 5].
\newblock Available from:
  \url{https://www.thebigquestions.com/2010/12/21/are-you-smarter-than-google/}.

\bibitem{Zare}
Zare D.
\newblock Google Question: In a Country in Which People Only Want Boys [answer
  on the Internet]. MathOverflow; 2010 Mar 12 [cited 2026 Aug 5].
\newblock Available from:
  \url{https://mathoverflow.net/questions/17960/google-question-in-a-country-in-which-people-only-want-boys}.

\bibitem{Griffiths}
Griffiths M.
\newblock The Carry-on-until-one-girl Proportion.
\newblock Math Gaz. 2015;99:464-7.

\bibitem{BrussPaindaveine}
Bruss FT, Paindaveine D.
\newblock Win Rates at First-Passage Times for Biased Simple Random Walks
  [preprint on the Internet]. arXiv; 2025 [cited 2026 Aug 5].
\newblock Available from: \url{https://arxiv.org/abs/2512.21254}.
\newblock arXiv:2512.21254.

\bibitem{GerholdHubalek}
Gerhold S, Hubalek F.
\newblock Unparalleled Instances of Prolifickness, Random Walks, and Square
  Root Boundaries [preprint on the Internet]. arXiv; 2025 [cited 2026 Aug 5].
\newblock Available from: \url{https://arxiv.org/abs/2504.06817}.
\newblock arXiv:2504.06817.

\bibitem{Propp}
Propp J.
\newblock Estimating {$\pi$} with a Coin [preprint on the Internet]. arXiv;
  2026 [cited 2026 Aug 5].
\newblock Available from: \url{https://arxiv.org/abs/2602.14487}.
\newblock arXiv:2602.14487.

\bibitem{AthreyaNey}
Athreya KB, Ney PE.
\newblock Branching Processes.
\newblock Berlin: Springer; 1972.

\bibitem{Robbins}
Robbins H.
\newblock A Note on Gambling Systems and Birth Statistics.
\newblock Am Math Mon. 1952;59:685-6.

\bibitem{Paseau}
Paseau AC.
\newblock Family Planning.
\newblock Math Gaz. 2011;95:213-7.

\bibitem{GrechJamesLauri}
Grech VE, James WH, Lauri J.
\newblock On Stopping Rules and the Sex Ratio at Birth.
\newblock Early Hum Dev. 2018;127:15-20.

\bibitem{Yamaguchi}
Yamaguchi K.
\newblock A Formal Theory for Male-Preferring Stopping Rules of Childbearing:
  Sex Differences in Birth Order and in the Number of Siblings.
\newblock Demography. 1989;26(3):451-65.

\bibitem{ChenLiShapiro}
Chen WYC, Li NY, Shapiro LW.
\newblock The Butterfly Decomposition of Plane Trees.
\newblock Discrete Appl Math. 2007;155(17):2187-201.

\bibitem{Cox}
Cox S.
\newblock Classifying Tree Topology Changes Along Tropical Line Segments.
\newblock Algebraic Stat. 2023;14(1):71-90.

\bibitem{Mohanty}
Mohanty SG.
\newblock On a Generalised Two-Coin Tossing Problem.
\newblock Biom Z. 1966;8(4):266-72.

\bibitem{JainConsul}
Jain GC, Consul PC.
\newblock A Generalized Negative Binomial Distribution.
\newblock SIAM J Appl Math. 1971;21(4):501-13.

\bibitem{Steutel}
Steutel FW.
\newblock First-Passage Times in a Skip-Free Random Walk.
\newblock Eindhoven: Technische Hogeschool Eindhoven; 1977. Memorandum COSOR
  77-23.

\bibitem{Dwass}
Dwass M.
\newblock The Total Progeny in a Branching Process and a Related Random Walk.
\newblock J Appl Probab. 1969;6(3):682-6.

\bibitem{Stanley}
Stanley RP.
\newblock Catalan Numbers.
\newblock Cambridge: Cambridge University Press; 2015.

\bibitem{Minami}
Minami N.
\newblock On the Number of Vertices with a Given Degree in a {Galton--Watson}
  Tree.
\newblock Adv Appl Probab. 2005;37(1):229-64.

\bibitem{KumarConsul}
Kumar A, Consul PC.
\newblock Negative Moments of a Modified Power Series Distribution and Bias of
  the Maximum Likelihood Estimator.
\newblock Commun Stat Theory Methods. 1979;8(2):151-66.

\end{thebibliography}

\AMMAuthorBiography

\end{document}